\documentclass[reqno,11pt,oneside]{amsart}
\usepackage{geometry}
\usepackage{amsmath, amsthm, amssymb, mathtools, mathrsfs, stmaryrd,bm}
\usepackage{amsfonts}
\usepackage{ascmac}
\usepackage{mathrsfs} 
\usepackage{xcolor}
\definecolor{HY}{RGB}{62,78,200}
\definecolor{HYH}{RGB}{105,90,190}
\usepackage[dvipdfmx, colorlinks=true, linkcolor=HY, citecolor=HYH, urlcolor=HY]{hyperref}
\usepackage{latexsym}
\usepackage{amssymb}
\address[Anju Yokoi]{Ikeda Senior High School Attached to Osaka Kyoiku University, 1-5-1, Midorigaoka, Ikeda-shi, Osaka, 563-0026, Japan}
\email{anju.scorpion@icloud.com}
\theoremstyle{definition}
\newtheorem{dfn}{Definition}[section]

\newtheorem{lem}[dfn]{Lemma}
\newtheorem{thm}[dfn]{Theorem}

\newtheorem{rem}[dfn]{Remark}

\makeatletter

\@addtoreset{equation}{section}
\makeatother
\allowdisplaybreaks
\title[A functional equation for multiple zeta functions]{A functional equation for multiple zeta functions\\
~and generalized confluent hypergeometric functions~}
\date{}
\author{Anju Yokoi}

\begin{document}
\maketitle
\begin{abstract}
  In this paper, we introduce a new function, the \emph{multiple confluent hypergeometric functions}, and establish a functional equation for the $r$-variable Euler--Zagier multiple zeta functions using it. In the case when $r=2$, this functional equation includes the well-known functional equation for the Euler--Zagier double zeta functions obtained by Matsumoto.
\end{abstract}
\section{Introduction}
\phantom{a}\\
Let $s_1,\ldots,s_r$ be complex variables, $i\coloneqq\sqrt{-1}$ and $(a)_n\coloneqq\frac{\Gamma(a+n)}{\Gamma(a)}$ in this paper. 
\phantom{a}\\
The $r$-variable Euler--Zagier sum is a kind of multiple zeta functions defined by the series
\begin{equation}
  \zeta_{EZ,r}(s_1,\ldots,s_r)\coloneqq\sum_{n_1,n_2,\ldots,n_r=1}^\infty n_1^{-s_1}(n_1+n_2)^{-s_2}\cdots(n_1+\cdots+n_r)^{-s_r},
\label{label12}\end{equation}
which is convergent absolutely when $\Re(s_{r-k+1}+\cdots+s_r)>k$ for $1\le k\le r$. When $r=1$, (\ref{label12}) is nothing but the Riemann zeta function.
The earliest result of the analytic continuation of (\ref{label12}) is due to Arakawa and Kaneko \cite{TM}, where the function is regarded as a one-variable function in $s_r$ only. However in the case $r=2$, Atkinson \cite{FV} analytically continued the function using the Poisson sum formula. Following these results, the analytic continuation of (\ref{label12}) as an $r$-variable meromorphic function has been established. (See \cite{SSY}, \cite{Mat2, Mat3} and \cite{M and S}.) In addition to these, various other methods for the analytic continuation of (\ref{label12}) have been established up to the present.
Now we have the following result.
\begin{thm}\label{thm3}
  The Euler--Zagier multiple zeta functions $\zeta_{EZ,r}(s_1,\ldots,s_r)$ can be meromorphically continued to $\mathbb{C}^r$ and has singularities on
  \[s_r=1,~s_{r-1}+s_r=2,1,0,-1,-2,-4\ldots\]
  and
  \[\sum_{i=1}^js_{r-i+1}\in\mathbb{Z}_{\le j}~,3\le j\le r\]
  where $\mathbb{Z}_{\le j}$ is the set of integers less than or equal to $j$.
\end{thm}
Once the Euler--Zagier multiple zeta functions in $r$-variables has been analytically continued, it is natural from a mathematical perspective to investigate its properties. In \cite{Mat1}, Matsumoto studied multiple zeta functions of Euler--Zagier type for $r=2$, and gave the functional equation for it, as the following theorem.
\begin{thm}[{\cite[Theorem 1]{Mat1}}]\label{thm1}
  Let $\Gamma(s)$ be the gamma function, respectively. We have
  \begin{equation}
    \frac{g(u,v)}{(2\pi)^{u+v-1}\Gamma(1-u)}=\frac{g(1-v,1-u)}{i^{u+v-1}\Gamma(v)}+2i\sin\left(\frac{\pi}{2}(u+v-1)\right)F_+(u,v),\notag
  \end{equation}
  where $u$ and $v$ are complex variables and $F_+(u,v)$ and $g(u,v)$ are defined by
  \begin{equation}
    F_+(u,v)\coloneqq\sum_{k=1}^\infty\sigma_{u+v-1}(k)\Psi(v,u+v;2\pi ik),
  \label{label9}\end{equation}
  and
  \begin{equation}
    g(u,v)\coloneqq\zeta_{EZ,2}(u,v)-\frac{\Gamma(1-u)}{\Gamma(v)}\Gamma(u+v-1)\zeta_{EZ,1}(u+v-1),\label{label20}
  \end{equation}
  by using the divisor sum function $\sigma_{s-1}(k)\coloneqq\sum_{d|k} d^{s-1}$ and the confluent hypergeometric function
\begin{equation}
  \Psi(a,c;x)\coloneqq\frac{1}{\Gamma(a)}\int_0^{\infty e^{i\phi}}e^{-xy}y^{a-1}(1+y)^{c-a-1}dy,\notag
\end{equation}
which is valid under the conditions $\Re a>0$, $-\pi<\phi<\pi$ and $|\phi+\arg x|<\pi/2$.
\end{thm}
According to \cite{Y.K.H.1}, since the second term on the right-hand side vanishes on the hyperplane 
\[
  \Omega_{2k+1}\coloneqq\{(s_1,s_2)\in\mathbb{C}^2\mid s_1+s_2=2k+1\}\quad(k\in\mathbb{Z}),
\]
 the above expression yields a beautiful symmetric form such as 
\begin{equation}
 \frac{1}{(2\pi)^{2k}\Gamma(1-s_1)}\zeta_{EZ,2}(s_1,s_2)=\frac{(-1)^k}{\Gamma(s_2)}\left\{\zeta_{EZ,2}(1-s_2,1-s_1)-\frac{B_{2k}}{4k}\right\}\label{equation1}\end{equation}
 
 when restricted to $\Omega_{2k+1}$. Here we denote the $k$-th Bernoulli number as $B_k$.

 Matsumoto also introduced the Mordell--Tornheim multiple zeta function
\begin{equation}
  \zeta_{MT,r}(s_1,\ldots,s_r;s_{r+1})\coloneqq\sum_{m_1,m_2,\ldots,m_r=1}^\infty  m_1^{-s_1}\cdots m_r^{-s_r}(m_1+\cdots+m_r)^{-s_{r+1}},\notag
\end{equation}
which is absolutely convergent in the following region,
\begin{equation}
  \sum_{\ell=1}^j\Re s_{k_\ell}+\Re s_{r+1}>j\notag
\end{equation}
with $1\le k_1<k_2<\cdots<k_j\le r$ for any $j=1,2\ldots,r$.
In addition, Okamoto and Onozuka \cite{Oka and Ono} derived the functional equation for the Mordell--Tornheim multiple zeta functions as the following theorem.
In \cite{Oka and Ono}, Okamoto and Onozuka prepared some functions to state their main theorem. More precisely, they introduced two divisor functions
\[
  \sigma_a(\ell_1,\ldots,\ell_r)\coloneqq\sum_{d|\ell_1,\ldots,d|\ell_r}d^a,
\]
and
 \begin{equation}
    \sigma_{MT,r}(s_1,\ldots,s_r,s_{r+1};\ell_1,\ldots,\ell_r)\coloneqq\sum_{d_1|\ell_1,\ldots,d_r|\ell_r}d_1^{s_1}\cdots d_r^{s_r}(d_1+\cdots+d_r)^{s_{r+1}}.\notag
  \end{equation}
  Furthermore, they put
 \begin{equation}
    g_r(s_1,\ldots,s_{r+1})\coloneqq\begin{multlined}[t]\zeta_{MT,r}(s_1,\ldots,s_r;s_{r+1})\\-\frac{\Gamma(1-s_r)\Gamma(s_r+s_{r+1}-1)}{\Gamma(s_{r+1})}\zeta_{MT,r-1}(s_1,\ldots,s_{r-1};s_r+s_{r+1}-1),\end{multlined}\label{label15}
  \end{equation}
  and
  \begin{equation}
    F_r^\pm(s_1,\ldots,s_{r+1})\coloneqq\begin{multlined}[t]\sum_{\ell_1,\ldots,\ell_{r-1}=1}^\infty \frac{\sigma_{s_1+\cdots+s_{r+1}-1}(\ell_1,\ldots,\ell_{r-1})}{\ell_1^{s_1}\cdots \ell_{r-1}^{s_{r-1}}}\\\times\Psi(s_{r+1},s_r+s_{r+1};\pm2\pi i(\ell_1+\cdots+\ell_{r-1})).\end{multlined}
  \label{label11}\end{equation}
\begin{thm}[{\cite[Theorem 1.2]{Oka and Ono}}]\label{thm2}
  We have
  \begin{align}
    \begin{multlined}[t]
     \frac{g_r(-s_1,\ldots,-s_{r-1},1-s_{r+1},1-s_r)}{i^{s_r+s_{r+1}-1}\Gamma(s_{r+1})}\\
     +e^{\frac{\pi i}{2}(s_r+s_{r+1}-1)}F_r^+(s_1,\ldots,s_{r+1})+e^{-\frac{\pi i}{2}(s_r+s_{r+1}-1)}F_r^-(s_1,\ldots,s_{r+1})
    \end{multlined}\notag\\
    =\begin{multlined}[t]
      \frac{g_r(s_1,\ldots,s_{r-1},s_r,s_{r+1})}{(2\pi)^{s_r+s_{r+1}-1}\Gamma(1-s_r)}\\
      +e^{-\frac{\pi i}{2}(s_r+s_{r+1}-1)}\sum_{\ell_1,\ldots,\ell_{r-1}=1}^\infty\sigma_{MT,r-1}(s_1,\ldots,s_{r-1},s_r+s_{r+1}-1;l_1,\ldots,\ell_{r-1})\\\times\left\{\Psi(s_{r+1},s_r+s_{r+1};2\pi i(\ell_1+\cdots+\ell_{r-1}))\right.\\\left.+\Psi(s_{r+1},s_r+s_{r+1};-2\pi i(\ell_1+\cdots+\ell_{r-1}))\right\}.
    \end{multlined}\notag
  \end{align}
\end{thm}
For two variables, the situation is as described above; however, no functional equation has yet been found for the Euler–Zagier multiple zeta functions when $r \ge 3$, a problem posed as an open question in \cite{Mat1}.
In the present paper, we discuss the fundamental idea based on Matsumoto's work, but with a different method of generalization, as in \cite{Oka and Ono}.
Here we introduce the space $\mathfrak{A}_r$ by
\begin{equation}
  \mathfrak{A}_r\coloneqq\{(s_1,\ldots,s_r)\mid,\;\Re s_{k+2}>1,\;1\le k\le r-2\}\subset\mathbb{C}^r\notag
\end{equation}
for $r>1$.
In that process we introduce the multiple confluent hypergeometric functions.
\begin{dfn}[Multiple confluent hypergeometric functions]
Let $a$ be a positive integer. We define the multiple confluent hypergeometric functions by the following infinite integral
\begin{align}
  &\Psi_a(h_1,\ldots,h_{a+1};x_1,\ldots,x_a;\delta)\notag\\
  &\coloneqq
  \begin{multlined}[t]\frac{1}{\Gamma(h_2)\cdots\Gamma(h_{a+1})}\int_0^{\infty e^{i\phi}} e^{-x_at_a}t_a^{h_{a+1}-1}\int_0^{\infty e^{i\phi}} e^{-x_{a-1}t_{a-1}}t_{a-1}^{h_a-1}
  \\\times\int_0^{\infty e^{i\phi}}\cdots\int_0^{\infty e^{i\phi}} e^{-x_2t_2}t_2^{h_3-1}\int_0^{\infty e^{i\phi}} e^{-x_1t_1}t_1^{h_2-1}(\delta+t_1+t_2+\cdots+t_a)^{h_1-1}dt_1\cdots dt_a,
  \end{multlined}\notag
  \end{align}
  where $0\le \delta\le 1$, complex variables $h_2,\ldots,h_{a+1}$ satisfy $\Re h_k>0$ for $2\le k\le a+1$ and $\phi$ satisfies $|\phi+\arg x|<\pi/2$.
\end{dfn}
We introduce two functions 
  \begin{equation*}
    \mathscr{F}_\pm^r(s_1,\ldots,s_r)\coloneqq \begin{multlined}[t]\sum_{k_1,\ldots,k_{r-1}=1}^\infty \sigma_{s_1+\cdots+s_r-1}(k_1,\ldots,k_{r-1})\\\times\Psi_{r-1}(s_1,\ldots,s_r;\pm2\pi ik_1,\pm2\pi i(k_1+k_2),\ldots,\pm2\pi i(k_1+\cdots+k_{r-1});1).\end{multlined}
  \end{equation*}
  which performs the same role as (\ref{label9}) and (\ref{label11}). This function is absolutely convergent when $\Re s_1<0$ and $\Re s_k>1$ for $2\le k\le r$. (See Theorem \ref{Theorem3}.) Moreover it can be continued meromorphically to $\mathfrak{A}_r$ space. (See Theorem \ref{thm2}.)
Furthermore, we put
 \begin{equation}\begin{multlined}[t] \mathscr{G}_r(s_1,\ldots,s_r)\coloneqq\zeta_{EZ,r}(s_1,\ldots,s_r)-\frac{1}{\Gamma(s_1)\cdots\Gamma(s_r)}\int_0^\infty\frac{t_r^{s_r-1}}{e^{t_r}-1}\int_0^\infty \cdots\int_0^\infty\frac{t_2^{s_2-1}}{e^{t_2+\cdots+t_r}-1}\\\times\int_0^\infty\frac{t_1^{s_1-1}}{t_1+t_2+\cdots+t_r}dt_1dt_2\cdots dt_r.\end{multlined}
 \label{div}\end{equation}
This function performs the same role as (\ref{label20}) and (\ref{label15}), and also it can be continued meromorphically to $\mathbb{C}^r$ space. (See Theorem \ref{thm4}.) Then we have the functional equation for (\ref{div}) in the whole $\mathfrak{A}_r$ space.
In order to continue $\mathscr{G}_r(s_1,\ldots,s_r)$, we introduce the well-known function, the Lauricella function, which is defined for $c\not\in\mathbb{Z}_{<0}$ by the series 
\begin{equation}
  F_D^{(N)}(\mathbf{a};b,c;\mathbf{z})\coloneqq\sum_{k_1,\ldots,k_N=0}^\infty\frac{(a_1)_{k_1}\cdots(a_N)_{k_N}(b)_{k_1+\cdots+k_N}}{(c)_{k_1+\cdots+k_N}k_1!\cdots k_N!}z_1^{k_1}\cdots z_N^{k_N},\notag
\end{equation}
which converges in the region $|z_k|<1$ for $1\le k\le N$. 
Finally, we introduce a kind of divisor function
\[\sigma_{EZ,r}(s_1,\ldots,s_r;k_1,k_2,\ldots,k_r)\coloneqq\sum_{d_1|k_1,\ldots,d_r|k_r}d_1^{s_1}(d_1+d_2)^{s_2}\cdots(d_1+d_2+\cdots+d_r)^{s_r}.\]
We call $\sigma_{EZ,r}$ by the Euler--Zagier $r$-divisor function.
\begin{thm}\label{main theorem}
Let $r$ be a positive integer satisfying $r\ge 2$. When the complex variables $(s_1,\ldots,s_r)$ are contained in $\mathfrak{A}_r$, we have 
  \begin{align}
  &\frac{\mathscr{G}_r(1-\textup{wt}(\bm{s})+s_1,1-\textup{wt}(\bm{s})+s_2,s_3,\ldots,s_r)}{\Gamma(\textup{wt}(\bm{s})-s_1)i^{\textup{wt}(\bm{s})-1}}+e^{\frac{\pi i}{2}(\textup{wt}(\bm{s})-1)}\mathscr{F}_+^r(s_1,\ldots,s_r)+e^{-\frac{\pi i}{2}(\textup{wt}(\bm{s})-1)}\mathscr{F}_-^r(s_1,\ldots,s_r)\notag\\
  &=\begin{multlined}[t]\frac{\mathscr{G}_r(s_1,\ldots,s_r)}{\Gamma(1-s_1)(2\pi)^{\textup{wt}(\bm{s})-1}}\\+e^{-\frac{\pi i}{2}(\textup{wt}(\bm{s})-1)}\sum_{k_1,\ldots,k_{r-1}=1}^\infty \sigma_{EZ,r-1}(2\textup{wt}(\bm{s})-s_1-s_2-1,-s_3,\ldots,-s_r;k_1,\ldots,k_r)\\\times\sum_{m_1,\ldots,m_{r-2}=0}^\infty\frac{(s_3)_{m_1}(1-\frac{k_1}{k_1+k_2})^{m_1}}{m_1!}\cdots\frac{(s_r)_{m_{r-2}}(1-\frac{k_1}{k_1+\cdots+k_{r-1}})^{m_{r-2}}}{m_{r-2}!}\\\times(\textup{wt}(\bm{s})-s_1)_{m_1+\cdots+m_{r-2}}\{\Psi(\textup{wt}(\bm{s})-s_1+m_1+\cdots+m_{r-2},\textup{wt}(\bm{s});2\pi ik_1;1)\\+\Psi(\textup{wt}(\bm{s})-s_1+m_1+\cdots+m_{r-2},\textup{wt}(\bm{s});-2\pi ik_1;1)\}.\end{multlined}\notag
\end{align}
Here we set $\bm{s}\coloneqq\{s_1,\ldots,s_r\}$ and define $\textup{wt}(\bm{s})=s_1+\cdots+s_r$. When all components are positive integers, $\bm{s}$ is called an index, and $\textup{wt}(\bm{s})$ is referred to as its weight.
\end{thm}
\begin{rem}
  The above theorem takes the form of an analogue of \cite[Theorem 1.2]{Oka and Ono}, and moreover it constitutes a generalization of \cite[Theorem 1]{Mat1}.
     In Theorem \ref{thm1}, the term $F_-(u,v)$ does not appear, and on the hyperplane we obtain a beautiful equation \ref{equation1}. However, when $r\ge 3$, such a phenomenon does not occur in the above Theorem; thus, it is evident that the result in Matsumoto's work is particularly elegant.
\end{rem}
\subsection{Examples.}
\phantom{a}\\
The following equation evidently holds by definition of the multiple confluent hypergeometric functions,
\begin{equation}
  \Psi_1(h_1,h_2;x_1;1)=\Psi(h_2,h_1+h_2;x_1).\notag
\end{equation}
Hence, it is straightforward to verify that Theorem \ref{thm1} holds by the Theorem \ref{main theorem} in the case when $r=2$.
Here we understand that $m_1,\ldots,m_{r-2}=0$ if $r=2$. 

In the case $r=3$ where we have
\begin{align}
  &\frac{\mathscr{G}_3(1-s_2-s_3,1-s_1-s_3,s_3)}{\Gamma(s_2+s_3)i^{s_1+s_2+s_3-1}}+e^{\frac{\pi}{2}(s_1+s_2+s_3-1)}\mathscr{F}_+^r(s_1,s_2,s_3)+e^{-\frac{\pi}{2}(s_1+s_2+s_3-1)}\mathscr{F}_-^r(s_1,s_2,s_3)\notag\\
  &=\begin{multlined}[t]\frac{\mathscr{G}_3(s_1,s_2,s_3)}{\Gamma(1-s_1)(2\pi)^{s_1+s_2+s_3-1}}+\sum_{k_1,k_2=1}^\infty \sigma_{EZ,2}(s_1+s_2+2s_3-1,-s_3;k_1,k_2)\\\times\sum_{m_1=0}^\infty\frac{(s_3)_{m_1}(1-\frac{k_1}{k_1+k_2})^{m_1}}{m_1!}(s_2+s_3)_{m_1}\{\Psi(s_2+s_3+m_1,s_1+s_2+s_3;2\pi ik_1;1)\\+\Psi(s_2+s_3+m_1,s_1+s_2+s_3;-2\pi ik_1;1)\}.\end{multlined}\notag
\end{align}
\section*{Acknowledgement}
The author would like to express his sincere gratitude to Prof.~Kohji Matsumoto and Prof.~Yayoi Nakamura for their many helpful pieces of advice and for holding a seminar on revising this paper for the author, and especially to Prof.~Kohji Matsumoto for his careful reading of the manuscript.
The author also wishes to thank Prof.~Shin-ichiro Seki for letting him know this open problem and encourage him attacking, Hanamichi Kawamura for teaching the author the basics of mathematics, and Prof.~Masanobu Kaneko for giving the author the opportunity to meet Prof.~Matsumoto.  
Thanks are also due to Yuushinn Saitou, Takumi Maesaka, Yuusuke Tabata, and Haruna Kai for their kind encouragement.  
This work  supported by the Academic Research Club of KADOKAWA DWANGO Educational Institute.  
Finally, the author wishes to thank all those who have supported him and his mathematics throughout this journey.
\section{Properties of multiple confluent hypergeometric functions}\label{sec2}
\begin{lem}\label{lem1}Let $n$ be a positive integer. We have
  \begin{align*}
    &\int_0^\infty x^{h_1-1}(1+x)^{h_2-1}(1+\alpha_1x)^{h_3-1}\cdots(1+\alpha_nx)^{h_{n+2}-1}dx\\
    &=\begin{multlined}[t]\frac{\Gamma(h_1)\Gamma(1-h_1-h_2-\cdots-h_{n+2}+n)}{\Gamma(1-h_2-h_3-\cdots-h_{n+2}+n)} \\\times F_D^{(n)}(1-h_3,\ldots,1-h_{n+2};h_1,1-h_2-h_3-\cdots-h_{n+2}+n;1-\alpha_1,\ldots,1-\alpha_n),
    \end{multlined}\notag
  \end{align*}
  where $h_j$, $\alpha_k$ are complex variables for $1\le j\le n+2$ and $1\le k\le n$.
\end{lem}
\begin{proof}
We prove it by the change of variables $x=\frac{t}{1-t}$ as follows
  \begin{align*}
    &\int_0^\infty x^{h_1-1}(1+x)^{h_2-1}(1+\alpha_1x)^{h_3-1}\cdots(1+\alpha_nx)^{h_{n+2}-1}dx\\
    &=\int_0^1 \left(\frac{t}{1-t}\right)^{h_1-1}\left(\frac{1}{1-t}\right)^{h_2-1}\left(1+\alpha_1\frac{t}{1-t}\right)^{h_3-1}\cdots\left(1+\alpha_n\frac{t}{1-t}\right)^{h_{n+2}-1}\frac{dt}{(1-t)^2}\\
    &=\int_0^1t^{h_1-1}(1-t)^{n-h_1-\cdots-h_{n+2}}(1-(1-\alpha_1)t)^{h_3-1}\cdots(1-(1-\alpha_n)t)^{h_{n+2}-1}dt\\
    &=\begin{multlined}[t]\sum_{m_1,\ldots,m_n=0}^\infty\frac{(1-h_3)_{m_1}(1-\alpha_1)^{m_1}}{m_1!}\cdots\frac{(1-h_{n+2})_{m_n}(1-\alpha_n)^{m_n}}{m_n!}\\\times\int_0^1t^{h_1+m_1+\cdots+m_n-1}(1-t)^{n-h_1-\cdots-h_{n+2}}dt\end{multlined}\\
    &=\begin{multlined}[t]\sum_{m_1,\ldots,m_n=0}^\infty\frac{(1-h_3)_{m_1}(1-\alpha_1)^{m_1}}{m_1!}\cdots\frac{(1-h_{n+2})_{m_n}(1-\alpha_n)^{m_n}}{m_n!}\\\times\frac{\Gamma(h_1+m_1+\cdots+m_n)\Gamma(1-h_1-h_2-\cdots-h_{n+2}+n)}{\Gamma(1-h_2-h_3-\cdots-h_{n+2}+m_1+\cdots+m_n+n)}.\end{multlined}
  \end{align*}
  We used the formula $\int_0^1t^{n-1}(1-t)^{m-1}dt=\frac{\Gamma(n)\Gamma(m)}{\Gamma(n+m)}$ in the last equality.
\end{proof}
\begin{lem} Let $a$ be a positive integer and let complex variables $h_1,\ldots,h_{a+1}$ satisfy $h_1<1$ and $0<h_2+\cdots+h_{a+1}$. We have
  \begin{align}
    &\Psi_a(h_1,\ldots,h_{a+1};x_1,\ldots,x_a;1)\notag\\&=\begin{multlined}[t]x_1^{h_3+\cdots+h_{a+1}}x_2^{-h_3}\cdots x_a^{-h_{a+1}}\sum_{m_1=0}^\infty\frac{(h_3)_{m_1}(1-\frac{x_1}{x_2})^{m_1}}{m_1!}\cdots\sum_{m_{a-1}=0}^\infty\frac{(h_{a+1})_{m_{a-1}}(1-\frac{x_1}{x_a})^{m_{a-1}}}{m_{a-1}!}\\\times(1-h_1)_{m_1+\cdots+m_{a-1}}\Psi(h_2+\cdots+h_{a+1}+m_1+\cdots+m_{a-1},h_1+\cdots+h_{a+1};x_1).\end{multlined}
  \label{label14}\end{align}
\end{lem}
\begin{proof}

By making the substitution $t_1=(\delta+t_2+\cdots+t_a)u$, the integral becomes
\begin{align}
  &\Psi_a(h_1,\ldots,h_{a+1};x_1,\ldots,x_a;\delta)\notag\\
  &=
  \begin{multlined}[t]\frac{1}{\Gamma(h_2)\cdots\Gamma(h_{a+1})}\int_0^\infty e^{-x_at_a}t_a^{h_{a+1}-1}\int_0^\infty e^{-x_{a-1}t_{a-1}}t_{a-1}^{h_a-1}
  \cdots\int_0^\infty e^{-x_2t_2}t_2^{h_3-1}\\\times\int_0^\infty e^{-x_1u(\delta+t_2+\cdots+t_a)}(\delta+t_2+\cdots+t_a)^{h_1+h_2-1}u^{h_2-1}(1+u)^{h_1-1}dudt_2\cdots dt_a
  \end{multlined}\notag\\
  &=\begin{multlined}[t]\frac{1}{\Gamma(h_3)\cdots\Gamma(h_{a+1})}\\\times\int_0^\infty e^{-x_at_a}t_a^{h_{a+1}-1}\int_0^\infty e^{-x_{a-1}t_{a-1}}t_{a-1}^{h_a-1}
  \cdots\int_0^\infty e^{-x_2t_2}t_2^{h_3-1}(\delta+t_2+\cdots+t_a)^{h_1+h_2-1}\\\times\Psi(h_2,h_1+h_2;x_1(\delta+t_2+\cdots+t_a))dt_2\cdots dt_a.
  \end{multlined}
  \label{label18}\end{align}
Recalling the well-known and beautiful property of $\Psi(b,c;x)$:
  \begin{equation}
    \Psi(b,c;x)=x^{1-c}\Psi(b-c+1,2-c;x)
  \end{equation}
  shown in \cite[6.5 (6)]{Bate}, we can show
  \begin{align}
  (\ref{label18})&= \begin{multlined}[t]\frac{x_1^{1-h_1-h_2}}{\Gamma(1-h_1)\Gamma(h_3)\cdots\Gamma(h_{a+1})}\int_0^\infty e^{-(x_a+x_1t_1)t_a}t_a^{h_{a+1}-1}\int_0^\infty e^{-(x_{a-1}+x_1t_1)t_{a-1}}t_{a-1}^{h_a-1}
  \\\times\int_0^\infty\cdots\int_0^\infty e^{-(x_2+x_1t_1)t_2}t_2^{h_3-1}\int_0^\infty e^{-\delta x_1t_1}t_1^{-h_1}(1+t_1)^{-h_2}dt_1\cdots dt_a
  \end{multlined}\notag\\
 &=\begin{multlined}[t]\frac{x_1^{1-h_1-h_2}}{\Gamma(1-h_1)}\int_0^\infty e^{-\delta x_1t_1}t_1^{-h_1}(1+t_1)^{-h_2}(x_2+x_1t_1)^{-h_3}\cdots(x_a+x_1t_1)^{-h_{a+1}}dt_1.
  \end{multlined}\label{label16}
  \end{align}
  Putting $\delta=1$, we see that the right-hand side of (\ref{label16}) is 
  \begin{align}
  &= \begin{multlined}[t]\frac{x_1^{1-h_1-h_2}x_2^{-h_3}\cdots x_a^{-h_{a+1}}}{2\pi i\Gamma(1-h_1)}\int_{\mathcal{M}}\Gamma(-s)x_1^s\\\times\int_0^\infty t_1^{s-h_1}(1+t_1)^{-h_2}\left(1+\left(\frac{x_1}{x_2}\right)t_1\right)^{-h_3}\cdots\left(1+\left(\frac{x_1}{x_a}\right)t_1\right)^{-h_{a+1}}dt_1ds.
  \end{multlined}\label{label1}
  \end{align}
  Here we used the formula $e^{-z}=\frac{1}{2\pi i}\int_{\mathcal{M}}\Gamma(-s)z^sds$.
  We define the integration contour $\mathcal{M}$ as the vertical line running from $c-i\infty$ to $c+i\infty$. In order to choose  $c$ so that all singularities of $\Gamma(-s)$ and $\Gamma(h_1+\cdots+h_{a+1}-s-1)$ lie to the right of the contour, we assume $\Re(h_1-1)<c<0$.
 Applying Lemma \ref{lem1} and summing over the poles of $\Gamma(-s)$ and $\Gamma(h_1+\cdots+h_{a+1}-s-1)$, we have
  \begin{align}
   (\ref{label1})&=\begin{multlined}[t]\frac{x_1^{1-h_1-h_2}x_2^{-h_3}\cdots x_a^{-h_{a+1}}}{2\pi i\Gamma(1-h_1)}\sum_{m_1,\ldots,m_{a-1}=0}^\infty\frac{(h_3)_{m_1}(1-\frac{x_1}{x_2})^{m_1}}{m_1!}\cdots\frac{(h_{a+1})_{m_{a-1}}(1-\frac{x_1}{x_a})^{m_{a-1}}}{m_{a-1}!}\\\times\int_{\mathcal{M}}\Gamma(-s)x_1^s\frac{\Gamma(s-h_1+m_1+\cdots+m_{a-1}+1)\Gamma(h_1+\cdots+h_{a+1}-s-1)}{\Gamma(h_2+h_3+\cdots+h_{a+1}+m_1+\cdots+m_{a-1})}ds\end{multlined}\notag\\
    &=\begin{multlined}[t]\frac{x_1^{1-h_1-h_2}x_2^{-h_3}\cdots x_a^{-h_{a+1}}}{\Gamma(1-h_1)}\sum_{m_1,\ldots,m_{a-1}=0}^\infty\frac{(h_3)_{m_1}(1-\frac{x_1}{x_2})^{m_1}}{m_1!}\cdots\frac{(h_{a+1})_{m_{a-1}}(1-\frac{x_1}{x_a})^{m_{a-1}}}{m_{a-1}!}\\\times\sum_{\ell=0}^\infty\frac{(-1)^\ell}{\ell!}\left(x_1^\ell\frac{\Gamma(\ell-h_1+m_1+\cdots+m_{a-1}+1)\Gamma(h_1+\cdots+h_{a+1}-\ell-1)}{\Gamma(h_2+h_3+\cdots+h_{a+1}+m_1+\cdots+m_{a-1})}
    \right.\\\left.
    +x_1^{h_1+\cdots+h_{a+1}-1+\ell}\frac{\Gamma(1-h_1-\cdots-h_{a+1}-\ell)\Gamma(h_2+h_3+\cdots+h_{a+1}+\ell+m_1+\cdots+m_{a-1})}{\Gamma(h_2+h_3+\cdots+h_{a+1}+m_1+\cdots+m_{a-1})}\right).\end{multlined}
 \label{label4} \end{align}
 By using the formula $\frac{1}{(1-x)_n}=(-1)^n(x)_{-n}$, we can show
  \begin{align}
  (\ref{label4})&=\begin{multlined}[t]\frac{x_1^{1-h_1-h_2}x_2^{-h_3}\cdots x_a^{-h_{a+1}}}{\Gamma(1-h_1)}\sum_{m_1,\ldots,m_{a-1}=0}^\infty\frac{(h_3)_{m_1}(1-\frac{x_1}{x_2})^{m_1}}{m_1!}\cdots\frac{(h_{a+1})_{m_{a-1}}(1-\frac{x_1}{x_a})^{m_{a-1}}}{m_{a-1}!}\\\times\sum_{\ell=0}^\infty\frac{1}{\ell!}\left(x_1^\ell\frac{\Gamma(1-h_1+m_1+\cdots+m_{a-1})\Gamma(h_1+\cdots+h_{a+1}-1)(1-h_1+m_1+\cdots+m_{a-1})_\ell}{(2-h_1-\cdots-h_{a+1})_\ell\Gamma(h_2+h_3+\cdots+h_{a+1}+m_1+\cdots+m_{a-1})}
    \right.\\\left.
    +x_1^{h_1+\cdots+h_{a+1}-1+\ell}\frac{\Gamma(1-h_1-\cdots-h_{a+1})(h_2+h_3+\cdots+h_{a+1}+m_1+\cdots+m_{a-1})_\ell}{(h_1+\cdots+h_{a+1})_\ell}\right)\end{multlined}\notag\\
    &=\begin{multlined}[t]\frac{x_1^{h_3+\cdots+h_{a+1}}x_2^{-h_3}\cdots x_a^{-h_{a+1}}}{\Gamma(1-h_1)}\sum_{m_1,\ldots,m_{a-1}=0}^\infty\frac{(h_3)_{m_1}(1-\frac{x_1}{x_2})^{m_1}}{m_1!}\cdots\frac{(h_{a+1})_{m_{a-1}}(1-\frac{x_1}{x_a})^{m_{a-1}}}{m_{a-1}!}\\\times
    \Gamma(1-h_1+m_1+\cdots+m_{a-1})\\\times\left(x_1^{1-h_1-\cdots-h_{a+1}}\frac{\Gamma(h_1+\cdots+h_{a+1}-1)}{\Gamma(h_2+\cdots+h_{a+1}+m_1+\cdots+m_{a-1})}{}_1F_1\left[\begin{aligned}
      1-&h_1+m_1+\cdots+m_{a-1}\\&2-h_1-\cdots-h_{a+1}
    \end{aligned};x_1\right] \right.\\\left.+\frac{\Gamma(1-h_1-\cdots-h_{a+1})}{\Gamma(1-h_1+m_1+\cdots+m_{a-1})}{}_1F_1\left[\begin{aligned}
      h_2+\cdots+&h_{a+1}+m_1+\cdots+m_{a-1}\\&h_1+\cdots+h_{a+1}
    \end{aligned};x_1\right] \right).\end{multlined}\label{label10}
  \end{align}
 The series representation 
  \begin{equation}
    \Psi(b,c,;x)=\frac{\Gamma(1-c)}{\Gamma(b-c+1)}{}_1F_1\left[\begin{aligned}
      &b\\&c
    \end{aligned};x\right]+x^{1-c}\frac{\Gamma(c-1)}{\Gamma(b)}{}_1F_1\left[\begin{aligned}
      b&-c+1\\&2-c
    \end{aligned};x\right]\label{label13}
  \end{equation}
  holds. (See \cite[6.5 (7)]{Bate}.) Here we define
  \begin{equation}
    {}_1F_1\left[\begin{aligned}
      &b\\&c
    \end{aligned};x\right]\coloneqq\sum_{m=0}^\infty\frac{(b)_m}{m!(c)_m}x^m.\notag
  \end{equation}
  Applying equation (\ref{label13}) for $b=h_2+\cdots+h_{a+1}+m_1+\cdots+m_{a-1}$, $c=h_1+\cdots+h_{a+1}$ and by equation (\ref{label10}) we arrive at the desired assertion.
  \end{proof}
   \begin{lem}\label{lem4}
    We have the asymptotic expansion of $\Psi_a(h_1,\ldots,h_a;x_1,\ldots,x_a)$ such as 
    \begin{align*}
&\Psi_a(h_1,\ldots,h_{a+1};x_1,\ldots,x_a)\\&=\begin{multlined}[t]x_1^{-h_2}\cdots x_a^{-h_{a+1}}\sum_{n=0}^{N-1}\left(\sum_{k_1+\cdots+k_a=n}\frac{(-x_1)^n(h_2)_{k_1}\cdots(h_{a+1})_{k_a}}{x_1^{k_1}\cdots x_a^{k_a}k_1!\cdots k_a!}\right)\frac{(1-h_1)_n}{x_1^n}\\+\rho_N(h_1,\ldots,h_{a+1};x_1,\ldots,x_a)\end{multlined}
    \end{align*}
    where, 
    \begin{equation}
    \rho_N(h_1,\ldots,h_{a+1};x_1,\ldots,x_a)\coloneqq  x_1^{1-h_1-h_2}x_2^{-h_3}\cdots x_a^{-h_{a+1}}\frac{1}{\Gamma(1-h_1)}\int_0^\infty e^{-x_1t}t^{-h_1}\int_0^t\frac{f(\eta)(t-\eta)^{N-1}}{(N-1)!}d\eta dt\notag
    \end{equation}
    and
    \begin{equation}
    \begin{multlined}[t]
      f(\eta;h_1,\ldots,h_{a+1};x_1,\ldots,x_a)\coloneqq\sum_{k_1+\cdots+k_a=N}\frac{N!}{k_1!\cdots k_a!}(-1)^{k_1}(h_2)_{k_1}(1+\eta)^{-h_2-k_1}\\\times\prod_{i=2}^a\left(-\left(\frac{x_1}{x_i}\right)\right)^{k_i}(h_{i+1})_{k_i}\left(1+\left(\frac{x_1}{x_i}\right)\eta\right)^{-h_{i+1}-k_i}.
       \end{multlined}\notag
    \end{equation}
  \end{lem}
  \begin{proof}
  We have the Taylor series
    \begin{align}
      &(1+t)^{-h_2}\left(1+\left(\frac{x_1}{x_2}\right)t\right)^{-h_3}\cdots\left(1+\left(\frac{x_1}{x_a}\right)t\right)^{-h_{a+1}}\notag\\&=\sum_{n=0}^{N-1}\left(\sum_{k_1+\cdots+k_a=n}\frac{(-x_1)^n(h_2)_{k_1}\cdots(h_{a+1})_{k_a}}{x_1^{k_1}\cdots x_a^{k_a}k_1!\cdots k_a!}\right)t^n+\int_0^t\frac{f(\eta)(t-\eta)^{N-1}}{(N-1)!}d\eta.\label{label17}
      \end{align}
    Applying the same method as in \cite{Bate} and using the equations (\ref{label16}) and (\ref{label17}), we have the desired equation.
  \end{proof}
\section{Properties of $\mathscr{G}_r(s_1,\ldots,s_r)$}\label{sec3}
It is easy to see that
\begin{equation}
  \zeta_{EZ,r}(s_1,\ldots,s_r)=\frac{1}{\Gamma(s_1)\cdots\Gamma(s_r)}\int_0^\infty\frac{t_r^{s_r-1}}{e^{t_r}-1}\int_0^\infty \cdots\int_0^\infty\frac{t_2^{s_2-1}}{e^{t_2+\cdots+t_r}-1}\int_0^\infty\frac{t_1^{s_1-1}}{e^{t_1+t_2+\cdots+t_r}-1}dt_1dt_2\cdots dt_r.\label{label2}
\end{equation}
The right-hand side is convergent when $\Re(s_{r-k+1}+\cdots+s_r)>k$ and $\Re(s_k)>0$ for $1\le k\le r$.

Let
\[h(z)\coloneqq\frac{1}{e^z-1}-\frac{1}{z}.\]
Applying (\ref{div}) and (\ref{label2}), we have 
\begin{equation*}
 \mathscr{G}_r(s_1,\ldots,s_r)=\frac{1}{\Gamma(s_1)\cdots\Gamma(s_r)}\int_0^\infty\frac{t_r^{s_r-1}}{e^{t_r}-1}\int_0^\infty \cdots\int_0^\infty\frac{t_2^{s_2-1}}{e^{t_2+\cdots+t_r}-1}\int_0^\infty t_1^{s_1-1}h(t_1+\cdots+t_r)dt_1dt_2\cdots dt_r.
\end{equation*}
 Let $\mathcal{C}$ be the contour which consists of the half-line on the positive real axis from
infinity to a small positive number, a small circle counterclockwise round the origin,
and the other half-line on the positive real axis back to infinity. Deforming the path
to the contour $\mathcal{C}$, we have
\begin{align*}
 &\mathscr{G}_r(s_1,\ldots,s_r)\\&=\frac{1}{\Gamma(s_1)\cdots\Gamma(s_r)(e^{2\pi is_1}-1)}\int_0^\infty\frac{t_r^{s_r-1}}{e^{t_r}-1}\int_0^\infty \cdots\int_0^\infty\frac{t_2^{s_2-1}}{e^{t_2+\cdots+t_r}-1}\int_\mathcal{C} t_1^{s_1-1}h(t_1+\cdots+t_r)dt_1dt_2\cdots dt_r.
 \end{align*}
 In \cite{KM}, the following estimate was proved:
 \[h(t_1+\cdots+t_r)=O(e^{-K|t_1+\cdots+t_r|}+(|t_1+\cdots+t_r|+1)^{-1}).\]
 The above estimate holds with a positive absolute constant $K$. Uniformly for any $x,y\in \mathcal{C}\cup[0,\infty)$, we can show
\begin{equation}
  \int_{\mathcal{C}}t_r^{s_1-1}h(t_1+\cdots+t_r)dt_r=O(1)\label{est1}
\end{equation}
 when $\Re s_1<1$. Assuming $\Re s>1$, we have
 \begin{equation}
   \int_0^\infty\frac{x^{s-1}}{e^x-1}dx=O(1).\label{est2}
 \end{equation}
 Finally, by applying (\ref{est1}) and (\ref{est2}),
  \[\mathscr{G}_r(s_1,\ldots,s_r)\ll1\]
  holds in the region $\Re s_1<1$ and $\Re s_k>1$ for $2\le k\le r$.
  \begin{thm}In the region $\Re s_1<0$ and $\Re s_k>1$ for $2\le k\le r$, we have 
    \begin{align}
      \mathscr{G}_r(s_1,\ldots,s_r)&=\begin{multlined}[t]\frac{(-1)^{s_1-1}\Gamma(1-s_1)}{\Gamma(s_2)\cdots\Gamma(s_r)}\sum_{n\not=0}\int_0^\infty\frac{t_r^{s_r-1}}{e^{t_r}-1}\int_0^\infty \cdots\int_0^\infty\frac{t_3^{s_3-1}}{e^{t_3+\cdots+t_r}-1}\\\times\int_0^\infty\frac{t_2^{s_2-1}}{e^{t_2+\cdots+t_r}-1}(-t_2-t_3\cdots-t_r+2\pi in)^{s_1-1}dt_2\cdots dt_r.\end{multlined}
    \end{align}
  \end{thm}
  \begin{proof}
  By the residue theorem, we have
  \begin{align}
    &\begin{multlined}[t]
      \int_0^{\frac{R}{2r}}\frac{t_r^{s_r-1}}{e^{t_r}-1}\int_0^{\frac{R}{2r}} \cdots\int_0^{\frac{R}{2r}}\frac{t_2^{s_2-1}}{e^{t_2+\cdots+t_r}-1}\int_{\mathcal{C}_R} t_1^{s_1-1}h(t_1+\cdots+t_r)dt_1dt_2\cdots dt_r\\
      +\int_0^{\frac{R}{2r}}\frac{t_r^{s_r-1}}{e^{t_r}-1}\int_0^{\frac{R}{2r}}\cdots\int_0^{\frac{R}{2r}}\frac{t_2^{s_2-1}}{e^{t_2+\cdots+t_r}-1}\int_{\mathcal{D}_R} t_1^{s_1-1}h(t_1+\cdots+t_r)dt_1dt_2\cdots dt_r
    \end{multlined}\notag\\
    &=-2\pi i\int_0^{\frac{R}{2r}}\frac{t_r^{s_r-1}}{e^{t_r}-1}\int_0^{\frac{R}{2r}}\cdots\int_0^{\frac{R}{2r}}\frac{t_2^{s_2-1}}{e^{t_2+\cdots+t_r}-1}\sum_{|n|\le N,n\not=0}(-t_2-\cdots-t_r+2\pi in)^{s_1-1}dt_2\cdots dt_r\label{integral}
  \end{align}
where $R=2\pi(N+\frac{1}{2})$ for a sufficiently large positive $N$.
    The contour $\mathcal{C}_R$ consists of the half-line on the positive real axis from $-(t_2+\cdots+t_r)+R$ to a small positive number, a small circle counterclockwise around the origin, and another half-line on the positive real axis back to $-(t_2+\cdots+t_r)+R$.
The contour $\mathcal{D}_R$ consists of a circle of radius $R$ centered at $-(t_2+\cdots+t_r)$, traversed clockwise once.
For the first term on the left-hand side of (\ref{integral}), the following  
\begin{multline}
   \int_0^{\frac{R}{2r}}\frac{t_r^{s_r-1}}{e^{t_r}-1}\int_0^{\frac{R}{2r}} \cdots\int_0^{\frac{R}{2r}}\frac{t_2^{s_2-1}}{e^{t_2+\cdots+t_r}-1}\int_{\mathcal{C}_R} t_1^{s_1-1}h(t_1+\cdots+t_r)dt_1dt_2\cdots dt_r\\
   \rightarrow \int_0^\infty \frac{t_r^{s_r-1}}{e^{t_r}-1}\int_0^\infty \cdots\int_0^\infty\frac{t_2^{s_2-1}}{e^{t_2+\cdots+t_r}-1}\int_\mathcal{C} t_1^{s_1-1}h(t_1+\cdots+t_r)dt_1dt_2\cdots dt_r\quad(R\rightarrow\infty)
\end{multline}
holds in the region $\Re s_1<1$ and $\Re s_k>1$ for $2\le k\le r$.
The second term on the left-hand side of (\ref{integral}) can be estimate as
\begin{multline}
  \int_0^{\frac{R}{2r}}\frac{t_r^{s_r-1}}{e^{t_r}-1}\int_0^{\frac{R}{2r}}\cdots\int_0^{\frac{R}{2r}}\frac{t_2^{s_2-1}}{e^{t_2+\cdots+t_r}-1}\int_{\mathcal{D}_R} t_1^{s_1-1}h(t_1+\cdots+t_r)dt_1dt_2\cdots dt_r\\\ll R^{\Re s_1}\rightarrow0\quad(R\rightarrow\infty)
\end{multline}
in the region $\Re s_1<0$ and $\Re s_k>1$ for $2\le k\le r$.
Lastly we consider the right-hand side of (\ref{integral}). In the region $\Re s_1<0$ and $\Re s_k>1$ for $2\le k\le r$, we obtain
\begin{multline}
  \int_0^\infty\frac{t_r^{s_r-1}}{e^{t_r}-1}\int_0^\infty\cdots\int_0^\infty\frac{t_2^{s_2-1}}{e^{t_2+\cdots+t_r}-1}\sum_{|n|\le N,n\not=0}(-t_2-\cdots-t_r+2\pi in)^{s_1-1}dt_2\cdots dt_r\\\ll\int_0^\infty\frac{t_r^{\Re s_r-1}}{e^{t_r}-1}\int_0^\infty\cdots\int_0^\infty\frac{t_2^{\Re s_2-1}}{e^{t_2+\cdots+t_r}-1}\sum_{n=1}^\infty|-t_2-\cdots-t_r+2\pi in|^{\Re s_1-1}dt_2\cdots dt_r\\\ll\int_0^\infty\frac{t_r^{\Re s_r-1}}{e^{t_r}-1}\int_0^\infty\cdots\int_0^\infty\frac{t_2^{\Re s_2-1}}{e^{t_2+\cdots+t_r}-1}dt_2\cdots dt_r\sum_{n=1}^\infty n^{\Re s_1-1}\ll1
\end{multline}
by (\ref{est2}). Hence we have
\begin{multline}
  -2\pi i\int_0^\infty\frac{t_r^{s_r-1}}{e^{t_r}-1}\int_0^\infty\cdots\int_0^\infty\frac{t_2^{s_2-1}}{e^{t_2+\cdots+t_r}-1}\sum_{|n|\le N,n\not=0}(-t_2-\cdots-t_r+2\pi in)^{s_1-1}dt_2\cdots dt_r\\\rightarrow-2\pi i\int_0^\infty\frac{t_r^{s_r-1}}{e^{t_r}-1}\int_0^\infty\cdots\int_0^\infty\frac{t_2^{s_2-1}}{e^{t_2+\cdots+t_r}-1}\sum_{n\not=0}(-t_2-\cdots-t_r+2\pi in)^{s_1-1}dt_2\cdots dt_r\quad(R\rightarrow\infty)
\end{multline}
in the region $\Re s_1<0$ and $\Re s_k>1$ for $2\le k\le r$.
Hence applying the formula
\[\frac{1}{\Gamma(s_1)(e^{2\pi is_1}-1)}=\frac{(-1)^{s_1}\Gamma(1-s_1)}{2\pi i},\]
we complete the proof.
  \end{proof}
  \begin{thm}\label{Theorem3}
    In the region $\Re s_1<0$ and $\Re s_k>1$ for $2\le k\le r$, we have
    \begin{align}
      \mathscr{G}_r(s_1,\ldots,s_r)&=\begin{multlined}[t](2\pi)^{s_1+\cdots+s_r-1}\Gamma(1-s_1)\{e^{\frac{\pi i(s_1+\cdots+s_r-1)}{2}}\mathscr{F}_+^r(s_1,\ldots,s_r)+e^{\frac{-\pi i(s_1+\cdots+s_r-1)}{2}}\mathscr{F}_-^r(s_1,\ldots,s_r)\}.\end{multlined}
    \end{align}
  \end{thm}
  \begin{rem}
    When $r=2$, Theorem \ref{Theorem3} is a special case of \cite[(2.14)]{Mat1}.
  \end{rem}
  \begin{proof}
By substituting $t_j=2\pi in\eta_j$, we can show
\begin{align*}
&\mathscr{G}_r(s_1,\ldots,s_r)\\
   &=\begin{multlined}[t]\frac{(2\pi i)^{s_1+\cdots+s_r-1}\Gamma(1-s_1)}{\Gamma(s_2)\cdots\Gamma(s_r)}\sum_{n=1}^\infty n^{s_1+\cdots+s_r-1}\sum_{m_1,\ldots,m_{r-1}=1}^\infty\int_0^{i\infty}e^{-2\pi inm_{r-1}\eta_r}\eta_r^{s_r-1}\\\times\int_0^{i\infty}\cdots\int_0^{i\infty}e^{-2\pi inm_2(\eta_3+\cdots+\eta_r)}\eta_3^{s_3-1}\int_0^{i\infty} e^{-2\pi inm_1(\eta_2+\cdots+\eta_r)}\eta_2^{s_2-1}(1+\eta_2+\eta_3+\cdots+\eta_r)^{s_1-1}d\eta_2\cdots d\eta_r\\
   +\frac{(-2\pi i)^{s_1+\cdots+s_r-1}\Gamma(1-s_1)}{\Gamma(s_2)\cdots\Gamma(s_r)}\sum_{n=1}^\infty n^{s_1+\cdots+s_r-1}\sum_{m_1,\ldots,m_{r-1}=1}^\infty\int_0^{-i\infty}e^{2\pi inm_{r-1}\eta_r}\eta_r^{s_r-1}\\\times\int_0^{-i\infty}\cdots\int_0^{-i\infty}e^{2\pi inm_2(\eta_3+\cdots+\eta_r)}\eta_3^{s_3-1}\int_0^{-i\infty} e^{2\pi inm_1(\eta_2+\cdots+\eta_r)}\eta_2^{s_2-1}(1+\eta_2+\eta_3+\cdots+\eta_r)^{s_1-1}d\eta_2\cdots d\eta_r
   \end{multlined}\\
    &=\begin{multlined}[t]\frac{(2\pi i)^{s_1+\cdots+s_r-1}\Gamma(1-s_1)}{\Gamma(s_2)\cdots\Gamma(s_r)}\sum_{n=1}^\infty n^{s_1+\cdots+s_r-1}\sum_{m_1,\ldots,m_{r-1}=1}^\infty\int_0^{i\infty} e^{-2\pi in(m_1+\cdots+m_{r-1})\eta_r}\eta_r^{s_r-1}\\\times\int_0^{i\infty }\cdots\int_0^{i\infty} e^{-2\pi in(m_1+m_2)\eta_3}\eta_3^{s_3-1}\int_0^{i\infty} e^{-2\pi inm_1\eta_2}\eta_2^{s_2-1}(1+\eta_2+\eta_3+\cdots+\eta_r)^{s_1-1}d\eta_2\cdots d\eta_r\\
    +\frac{(-2\pi i)^{s_1+\cdots+s_r-1}\Gamma(1-s_1)}{\Gamma(s_2)\cdots\Gamma(s_r)}\sum_{n=1}^\infty n^{s_1+\cdots+s_r-1}\sum_{m_1,\ldots,m_{r-1}=1}^\infty\int_0^{-i\infty} e^{2\pi in(m_1+\cdots+m_{r-1})\eta_r}\eta_r^{s_r-1}\\\times\int_0^{-i\infty }\cdots\int_0^{-i\infty} e^{2\pi in(m_1+m_2)\eta_3}\eta_3^{s_3-1}\int_0^{-i\infty} e^{2\pi inm_1\eta_2}\eta_2^{s_2-1}(1+\eta_2+\eta_3+\cdots+\eta_r)^{s_1-1}d\eta_2\cdots d\eta_r\end{multlined}\\
    &=\begin{multlined}[t](2\pi)^{s_1+\cdots+s_r-1}e^{\frac{\pi i(s_1+\cdots+s_r-1)}{2}}\Gamma(1-s_1)\mathscr{F}_+^r(s_1,\ldots,s_r)+(2\pi)^{s_1+\cdots+s_r-1}e^{\frac{\pi i(1-s_1-\cdots-s_r)}{2}}\Gamma(1-s_1)\mathscr{F}_-^r(s_1,\ldots,s_r).\end{multlined}
\end{align*}
This complete the proof.
\end{proof}
 \section{Analytic continuation of $\mathscr{F}_\pm^r(s_1,\ldots,s_r)$ and $\mathscr{G}_r(s_1,\ldots,s_r)$}
  Here we introduce a new function. We put
  \begin{equation}
    \zeta_{EZ,r}(s_1,\ldots,s_r;f)\coloneqq\sum_{m_1,\ldots,m_r=1}^\infty\frac{f(m_1,\ldots,m_r)}{m_1^{s_1}(m_1+m_2)^{s_2}\cdots(m_1+\cdots+m_r)^{s_r}},\notag
  \end{equation}
  where $f$ is a multi-variable arithmetic function.
  \begin{lem}\label{Lemma 4.1}
    Let $a$ be a complex variable. Then, in the region $\Re s_1+\Re s_2+\cdots+\Re s_r>\Re a+1$ and $\Re(s_{r-k+1}+\cdots+s_r)>k$ for $1\le k\le r$, we have
    \begin{equation}
      \zeta_{EZ,1}(\textup{wt}(\bm{s})-a)\zeta_{EZ,r}(s_1,\ldots,s_r)=\zeta_{EZ,r}(s_1,\ldots,s_r;\sigma_a).\notag
    \end{equation}
  \end{lem}
  \begin{proof}In the above region, $\zeta_{EZ,1}(\textup{wt}(\bm{s})-a)$ and $\zeta_{EZ,r}(s_1,\ldots,s_r)$ are absolutely convergent.
 We can show
    \begin{align*}
      \zeta_{EZ,1}(s-a)\zeta_{EZ,r}(s_1,\ldots,s_r)&=\sum_{n,m_1,\ldots,m_r=1}^\infty\frac{n^a}{(nm_1)^{s_1}(nm_1+nm_2)^{s_2}\cdots(nm_1+\cdots+nm_r)^{s_r}}\\
      &=\sum_{\ell_1,\ldots,\ell_r=1}^\infty\frac{\sum_{d|\ell_1,\ldots,d|\ell_r}d^a}{\ell_1^{s_1}(\ell_1+\ell_2)^{s_2}\cdots(\ell_1+\cdots+\ell_r)^{s_r}}.
    \end{align*}
    Here we obtain the proof.
  \end{proof}
 \begin{thm}\label{thm2}
The function $\mathscr{F}_\pm^r(s_1,\ldots,s_r)$ can be continued meromorphically to the whole $\mathfrak{A}_r$ space.
 \end{thm}
 \begin{proof}
 First we assume $\Re s_1<0$ and $\Re s_k>1$ for $2\le k\le r$. Applying Lemmas \ref{lem4} and \ref{Lemma 4.1}, we have
 \begin{align*}
   & \mathscr{F}_\pm^r(s_1,\ldots,s_r)\\&=\begin{multlined}[t]
    \sum_{k_1,\ldots,k_{r-1}=1}^\infty \sigma_{s_1+\cdots+s_r-1}(k_1,\ldots,k_{r-1})\\\times\Bigg(\sum_{n=0}^{N-1}\sum_{m_1+\cdots+m_{r-1}=n}(\pm2\pi i)^{-s_2-\cdots-s_r-n}\frac{(-1)^nk_1^{m_2+\cdots+m_{r-1}}(s_2)_{m_1}\cdots(s_r)_{m_{r-1}}}{(k_1+k_2)^{m_2}\cdots (k_1+\cdots+k_{r-1})^{m_{r-1}}m_1!\cdots m_{r-1!}}\\\times\frac{(1-s_1)_n}{k_1^{s_2+n}(k_1+k_2)^{s_3}\cdots(k_1+\cdots+k_{r-1})^{s_r}}\\+\rho_N(s_1,\ldots,s_r;\pm2\pi ik_1,\pm2\pi i(k_1+k_2)\ldots,\pm2\pi i(k_1+\cdots+k_{r-1}))\Bigg)
    \end{multlined}\\
    &=\begin{multlined}[t]
    \sum_{n=0}^{N-1}(\pm2\pi i)^{-s_2-\cdots-s_r-n}(1-s_1)_n\sum_{m_1+\cdots+m_{r-1}=n}\frac{(-1)^n(s_2)_{m_1}\cdots(s_r)_{m_{r-1}}}{m_1!\cdots m_{r-1!}}
    \\\times\sum_{k_1,\ldots,k_{r-1}=1}^\infty \sigma_{s_1+\cdots+s_r-1}(k_1,\ldots,k_{r-1})\frac{1}{k_1^{s_2+m_1}(k_1+k_2)^{s_3+m_2}\cdots(k_1+\cdots+k_{r-1})^{s_r+m_{r-1}}}\\+\sum_{k_1,\ldots,k_{r-1}=1}^\infty \sigma_{s_1+\cdots+s_r-1}(k_1,\ldots,k_{r-1})\rho_N(s_1,\ldots,s_r;\pm2\pi ik_1,\pm2\pi i(k_1+k_2),\ldots,\pm2\pi i(k_1+\cdots+k_{r-1}))
    \end{multlined}\\
     &=\begin{multlined}[t]
    (\pm2\pi i)^{-s_2-\cdots-s_r}\sum_{n=0}^{N-1}(1-s_1)_n\sum_{m_1+\cdots+m_{r-1}=n}\frac{(-1)^n(s_2)_{m_1}\cdots(s_r)_{m_{r-1}}}{m_1!\cdots m_{r-1!}}
    \\\times\zeta_{EZ,1}(n-s_1+1)\zeta_{EZ,r-1}(s_2+m_1,s_3+m_2,\ldots,s_r+m_{r-1})\\+\sum_{k_1,\ldots,k_{r-1}=1}^\infty \sigma_{s_1+\cdots+s_r-1}(k_1,\ldots,k_{r-1})\rho_N(s_1,\ldots,s_r;\pm2\pi ik_1,\pm2\pi i(k_1+k_2),\ldots,\pm2\pi i(k_1+\cdots+k_{r-1})).
    \end{multlined}
 \end{align*}
Now we estimate $\rho_N$. We can show
\begin{align}
&\rho_N(s_1,\ldots,s_r;\pm2\pi ik_1,\ldots,\pm2\pi i(k_1+\cdots+k_{r-1}))\notag\\
&=\begin{multlined}[t](\pm2\pi ik_1)^{1-s_1-s_2}(\pm2\pi i(k_1+k_2))^{-s_3}\cdots(\pm2\pi i(k_1+\cdots+k_{r-1}))^{-s_r}\\\times\frac{1}{\Gamma(1-s_1)}\int_0^\infty e^{-(\pm2\pi ik_1)t}t^{-s_1}\\\times\int_0^t\frac{f(\eta;s_1,\ldots,s_r;2\pi ik_1,\ldots,\pm2\pi i(k_1+\cdots+k_{r-1}))(t-\eta)^{N-1}}{(N-1)!}d\eta dt\end{multlined}\notag\\
&=\begin{multlined}[t](\pm2\pi ik_1)^{1-s_1-s_2}(\pm2\pi i(k_1+k_2))^{-s_3}\cdots(\pm2\pi i(k_1+\cdots+k_{r-1}))^{-s_r}\\\times\frac{1}{\Gamma(1-s_1)}\sum_{m_1+\cdots+m_{r-1}=N}\frac{N}{m_1!\cdots m_{r-1}!}(-1)^{m_1}(s_2)_{m_1}\int_0^\infty e^{-(\pm2\pi ik_1)t}t^{-s_1}\\\times\int_0^t(1+\eta)^{-s_2-m_1}\prod_{j=2}^{r-1}\left(-\frac{k_1}{k_1+\cdots+k_j}\right)^{m_j}(s_{j+1})_{m_j}\\\times\left(1+\left(\frac{k_1}{k_1+\cdots+k_j}\right)\eta\right)^{-s_{j+1}-m_j}(t-\eta)^{N-1}d\eta dt\end{multlined}\notag\\
&=\begin{multlined}[t](\pm2\pi ik_1)^{1-s_1-s_2}(\pm2\pi i(k_1+k_2))^{-s_3}\cdots(\pm2\pi i(k_1+\cdots+k_{r-1}))^{-s_r}\\\times\frac{1}{\Gamma(1-s_1)}\sum_{m_1+\cdots+m_{r-1}=N}\frac{N}{m_1!\cdots m_{r-1}!}(-1)^{m_1}(s_2)_{m_1}\int_0^\infty e^{-(\pm2\pi ik_1)t}t^{-s_1+N-1}\\\times\int_0^t\left(1+\eta\right)^{-s_2-m_1}\prod_{j=2}^{r-1}\left(-\frac{k_1}{k_1+\cdots+k_j}\right)^{m_j}(s_{j+1})_{m_j}\\\times\left(1+\left(\frac{k_1}{k_1+\cdots+k_j}\right)\eta\right)^{-s_{j+1}-m_j}(1-\eta/t)^{N-1}d\eta dt.\end{multlined}\label{label7}
\end{align}

Putting $\frac{\eta}{t}=\xi$, and then $\pm2\pi ik_1t=\mu$, we have

\begin{align*}
  (\ref{label7})&=\begin{multlined}[t](\pm2\pi ik_1)^{-s_2-N}(\pm2\pi i(k_1+k_2))^{-s_3}\cdots(\pm2\pi i(k_1+\cdots+k_{r-1}))^{-s_r}\\\times\frac{1}{\Gamma(1-s_1)}\sum_{m_1+\cdots+m_{r-1}=N}\frac{N}{m_1!\cdots m_{r-1}!}(-1)^{m_1}(s_2)_{m_1}
\int_0^{\pm i\infty} e^{-\mu}\mu^{-s_1+N}\\\times\int_0^1\left(1\pm\frac{\xi\mu}{2\pi ik_1} \right)^{-s_2-m_1}\prod_{j=2}^{r-1}\left(-\frac{k_1}{k_1+\cdots+k_j}\right)^{m_j}(s_{j+1})_{m_j}\\\times\left(1\pm\left(\frac{k_1}{k_1+\cdots+k_j}\right)\frac{\xi\mu}{2\pi ik_1}\right)^{-s_{j+1}-m_j} (1-\xi)^{N-1}d\xi d\mu.\end{multlined}
\end{align*}
According to \cite{Mat2} if $\Re s_{j+1}\ge 0$, we have
\begin{equation}
  \left|\left(1\pm\left(\frac{k_1}{k_1+\cdots+k_j}\right)\frac{\xi\mu}{2\pi ik_1}\right)^{-s_{j+1}-m_i}\right|\le e^{\frac{\pi|\Im s_{j+1}|}{2}}.\notag
\end{equation}
Hence, using the fact and assuming $\Re s_1<N+1$, we obtain
\begin{align*}
  &|\rho_N(s_1,\ldots,s_r;\pm2\pi ik_1,\ldots,\pm2\pi i(k_1+\cdots+k_{r-1}))|\\&\le\begin{multlined}[t](2\pi k_1)^{-\Re s_2-N}(2\pi (k_1+k_2))^{-\Re s_3}\cdots(2\pi (k_1+\cdots+k_{r-1}))^{-\Re s_r}\\\times\frac{1}{|\Gamma(1-s_1)|}\sum_{m_1+\cdots+m_{r-1}=N}\frac{\Gamma(N-\Re s_1+1)}{m_1!\cdots m_{r-1}!}|(s_2)_{m_1}|\\\times
\prod_{j=2}^{r-1}\left(\frac{k_1}{k_1+\cdots+k_j}\right)^{m_j}|(s_{j+1})_{m_j}|e^{\pi(|\Im s_2|+\cdots+|\Im s_r|)} .\end{multlined}
\end{align*}
Finally, we have 
\begin{align*}
&\sum_{k_1,\ldots,k_{r-1}=1}^\infty \left|\sigma_{s_1+\cdots+s_r-1}(k_1,\ldots,k_{r-1})\rho_N(s_1,\ldots,s_r;\pm2\pi ik_1,\pm2\pi i(k_1+k_2),\ldots,\pm2\pi i(k_1+\cdots+k_{r-1}))\right|\\
\notag&\le \begin{multlined}[t]\frac{(2\pi)^{-\Re s_2-\cdots-\Re s_r-N}}{|\Gamma(1-s_1)|}e^{\pi(\Im s_2+\cdots+\Im s_r)}\sum_{m_1+\cdots+m_{r-1}=N}\frac{\Gamma(N-\Re s_1+1)}{m_1!\cdots m_{r-1}!}
\prod_{j=1}^{r-1}|(s_{j+1})_{m_j}|\\\times\sum_{k_1,\ldots,k_{r-1}=1}^\infty \frac{\sigma_{\Re(s_1+\cdots+s_r)-1}(k_1,\ldots,k_{r-1})}{k_1^{\Re s_2+m_1}(k_1+k_2)^{\Re s_3+m_2}\cdots(k_1+\cdots+k_{r-1})^{\Re s_r+m_{r-1}}}\end{multlined}
\end{align*}
and applying Lemma \ref{Lemma 4.1} we see that the last sum is equal to $\zeta_{EZ,1}(1-\Re s_1+N)\zeta_{EZ,r-1}(\Re s_2+m_1,\Re s_3+m_2,\ldots,\Re s_r+m_{r-1})$ and convergent absolutely when $\Re s_1<N$, $\Re s_2+\cdots+\Re s_r+N>r-1$ and $\Re (s_{r-k+1}+\cdots+s_r)>k$ for $1\le k\le r-2$.
Since $N$ is arbitrary, $\mathscr{F}_\pm^r(s_1,\ldots,s_r)$ can be continued meromorphically to the whole $\mathfrak{A}_r$ space.
\end{proof}
 Applying the equation
  \begin{equation}
    \Psi(b,c;x)=x^{1-c}\Psi(b-c+1,2-c;x),\notag
  \end{equation}
 and putting $\ell_j=dk_j$, we have
  \begin{align}
    &\mathscr{F}_\pm^r(s_1,\ldots,s_r)\notag\\&= (\pm2\pi i)^{1-s_1-\cdots-s_r}\begin{multlined}[t]\sum_{\ell_1,\ell_2,\ldots,\ell_{r-1}=1}^\infty \sigma_{EZ,r-1}(1-s_1-s_2,-s_3,\ldots,-s_r;\ell_1,\ell_2,\ldots,\ell_{r-1})\\\times\sum_{m_1=0}^\infty\frac{(s_3)_{m_1}(1-(\frac{\ell_1}{\ell_1+\ell_2}))^{m_1}}{m_1!}\cdots\sum_{m_{r-2}=0}^\infty\frac{(s_r)_{m_{r-2}}(1-(\frac{\ell_1}{\ell_1+\cdots+\ell_{r-1}}))^{m_{r-2}}}{m_{r-2}!}\\\times(1-s_1)_{m_1+\cdots+m_{r-2}}\Psi(1-s_1+m_1+\cdots+m_{r-2},2-s_1-\cdots-s_r;\pm2\pi i\ell_1).\end{multlined}\label{label5}
 \end{align}
 Here we introduce a new multiple zeta functions $\zeta_{A,r}(s_1,\ldots,s_{r-1};t_1,\ldots,t_{r-1};s_r)$ by the series
  \begin{align*}
    \zeta_{A,r}(s_1,s_2,\ldots,s_r;t_1,t_2\ldots,t_{r-1})\coloneqq\begin{multlined}[t]\sum_{n_1,n_2,n_3\ldots,n_r=1}^\infty\frac{1}{n_1^{s_1}(n_1+n_2)^{s_2}\cdots(n_1+\cdots+n_r)^{s_r}}\\\times\frac{1}{n_2^{t_1}(n_2+n_3)^{t_2}\cdots(n_2+\cdots+n_r)^{t_{r-1}}}.\end{multlined}
  \end{align*}
  \begin{rem}
  Let $\mathfrak{g}$ be a complex semisimple Lie algebra of rank $r$. We denote by $\Delta=\Delta(\mathfrak{g})$ the set of all roots of $\mathfrak{g}$, by $\Delta_+=\Delta_+(\mathfrak{g})$ the set of all positive roots of $\mathfrak{g}$.
  Here we denote the zeta functions of the root system by the series
  \begin{equation} \zeta_r(\bm{s};\Delta)\coloneqq\sum_{m_1,\ldots,m_r=1}^\infty\prod_{\alpha\in\Delta_+}\langle \alpha^{\vee},m_1\lambda_1+\cdots+m_r\lambda_r\rangle^{-s_\alpha},\notag
  \end{equation}
  where $\bm{s}=(s_\alpha)_{\alpha\in\Delta_+}\in\mathbb{C}^n$ and $n=|\Delta_+|$ is the number of positive roots of $\mathfrak{g}$.
  When $\Delta$ corresponds to the root system of type $A_r$, $\zeta_r(\bm{s};\Delta)$ is sometimes written as $\zeta_r(\bm{s};A_r)$.
In the case $r=2$, we have $\zeta_{A,2}(s_1,s_2;t_1)=\zeta_{MT,2}(s_1,t_1;s_2)=\zeta_2(s_1,t_1,s_2;A_2)$. In the case $r=3$, we have $\zeta_{A,3}(s_1,s_2,s_3;t_1,t_2)=\zeta_3(s_1,t_1,0,s_2,t_2,s_3;A_3)$.
Similarly, it is easy to see that a general $\zeta_{A,r}(s_1,\ldots,s_{r-1};t_1,\ldots,t_{r-1};s_r)$  can also be expressed in terms of $\zeta_r(\bm{s};A_r)$.
Since $\zeta_r(\bm{s};\Delta)$ can be analytically continued to the whole $\mathbb{C}^n$ space, according to \cite{Y.K.H.2}, and since $\zeta_{A,r}(s_1,\ldots,s_{r-1};t_1,\ldots,t_{r-1};s_r)$ is a special case of $\zeta_r(\bm{s};\Delta)$, it can also be analytically continued to $\mathbb{C}^{2r-1}$ space.
\end{rem}
\begin{lem}\label{lem2}
We have
\[\zeta_{A,r}(s_1,\ldots,s_r;t_1,\ldots,t_{r-1})=\zeta_r(\bm{z};A_r).\] 
Here we set $\bm{z}$ by
\begin{equation*}
  \bm{z}\coloneqq \left\{ (z_1,\ldots,z_n)\in\mathbb{C}^n \middle|
  \begin{aligned} 
    &\;z_j=s_k\;\textup{if}\;j=1+\frac{(k-1)(2r-k+2)}{2}\;\textup{for}\;1\le k\le r 
     \\
     &\;z_j=t_k\;\textup{if}\;j=2+\frac{(k-1)(2r-k+2)}{2}\;\textup{for}\;1\le k\le r-1\\
    &\;z_j=0,\;\textup{otherwise}
  \end{aligned}
  \right\}
\end{equation*}
for $1\le j\le n$.
\end{lem}

 \begin{thm}\label{thm4}
   The function $\mathscr{G}_r(s_1,\ldots,s_r)$ can be continued meromorphically to the whole $\mathbb{C}^r$ space.
 \end{thm}
 \begin{proof}
 By substituting $t_1=(t_2+\cdots+t_r)\eta$, we have
   \begin{align}
     &\frac{1}{\Gamma(s_1)\cdots\Gamma(s_r)}\int_0^\infty\frac{t_r^{s_r-1}}{e^{t_r}-1}\int_0^\infty\cdots\int_0^\infty\frac{t_2^{s_2-1}}{e^{t_2+\cdots+t_r}-1}\int_0^\infty\frac{t_1^{s_1-1}}{t_1+\cdots+t_r}dt_1\cdots dt_r\notag\\
     &=\begin{multlined}[t]\frac{\Gamma(1-s_1)}{\Gamma(s_2)\cdots\Gamma(s_r)}\sum_{n_1,\ldots,n_{r-1}=1}^\infty\int_0^\infty e^{-(n_1+\cdots+n_{r-1})t_r}t_r^{s_r-1}\\\times\int_0^\infty\cdots\int_0^\infty e^{-n_1t_2}t_2^{s_2-1}(t_2+\cdots+t_r)^{s_1-1}dt_2\cdots dt_r\end{multlined}\notag\\
    &=\Gamma(1-s_1)\sum_{n_1,\ldots,n_{r-1}=1}^\infty \Psi_{r-1}(s_1,\ldots,s_r;n_1,\ldots,n_1+\cdots+n_{r-1};0).\label{label21}
   \end{align}
   Applying equation (\ref{label16}), we have
   \begin{align}
   (\ref{label21})&=\begin{multlined}[t]\sum_{n_1,\ldots,n_{r-1}=1}^\infty n_1^{1-s_1-s_2}(n_1+n_2)^{-s_3}\cdots(n_1+\cdots+n_{r-1})^{-s_r}\\\times\int_0^\infty t^{-s_1}(1+t)^{-s_2}\left(1+\frac{n_1}{n_1+n_2}t\right)^{-s_3}\cdots\left(1+\frac{n_1}{n_1+\cdots+n_{r-1}}t\right)^{-s_r}dt.\end{multlined}\label{label22}
   \end{align}
   Hence, by Lemma \ref{lem1} we have
   \begin{align}
   (\ref{label22})&=\begin{multlined}[t] \frac{\Gamma(1-s_1)\Gamma(s_1+\cdots+s_r-1)}{\Gamma(s_2+\cdots+s_r)}\sum_{n_1,\ldots,n_{r-1}=1}^\infty n_1^{1-s_1-s_2}(n_1+n_2)^{-s_3}\cdots(n_1+\cdots+n_{r-1})^{-s_r}\\\times F_D^{(r-2)}\left(s_3,\ldots,s_r;1-s_1,s_2+\cdots+s_r;\frac{n_2}{n_1+n_2},\ldots,\frac{n_2+\cdots+n_{r-1}}{n_1+\cdots+n_{r-1}}\right).\label{label23}\end{multlined}
   \end{align}
   According to \cite{SI}, the Lauricella function has the Mellin--Barnes integral representation
   \begin{align}
     F_D^{(N)}(\mathbf{a};b,c;\mathbf{z})=\begin{multlined}[t]\frac{\Gamma(c)}{(2\pi i)^N\Gamma(b)\Gamma(a_1)\cdots\Gamma(a_N)}\\
     \times\int_{\mathcal{L}_1}\cdots\int_{\mathcal{L}_N}\frac{\Gamma(b+t_1+\cdots+t_N)}{\Gamma(c+t_1+\cdots+t_N)}\left(\prod_{j=1}^N\Gamma(a_j+t_j)\Gamma(-t_j)(-z_j)^{t_j}\right)dt_1\cdots dt_N,\end{multlined}\notag
   \end{align}
   where $\mathcal{L}_j$ is a contour in the $t_j$-plane which is a deformed imaginary axis, that is, it connects $-i\infty$ and $+i\infty$ but is curved so that among all the poles of the integrand only the poles of $\Gamma(-t_j)$ lie to the right of $\mathcal{L}_j$. (See \cite[Equation (1.7)]{SI}.)
   We have
   \begin{align}
     (\ref{label23})&=\begin{multlined}[t]
     \frac{\Gamma(1-s_1)\Gamma(s_1+\cdots+s_r-1)}{\Gamma(s_2+\cdots+s_r)}\sum_{n_1,\ldots,n_{r-1}=1}^\infty n_1^{1-s_1-s_2}(n_1+n_2)^{-s_3}\cdots(n_1+\cdots+n_{r-1})^{-s_r}\\\times
     \frac{\Gamma(s_2+\cdots+s_r)}{(2\pi i)^{r-2}\Gamma(1-s_1)\Gamma(s_3)\cdots\Gamma(s_r)}
     \int_{\mathcal{L}_1}\cdots\int_{\mathcal{L}_{r-2}}\frac{\Gamma(1-s_1+t_1+\cdots+t_{r-2})}{\Gamma(s_2+\cdots+s_r+t_1+\cdots+t_{r-2})}\\\times\left(\prod_{j=1}^{r-2}\Gamma(s_{j+2}+t_j)\Gamma(-t_j)\left(-\frac{n_2+\cdots+n_{j+1}}{n_1+\cdots+n_{j+1}}\right)^{t_j}\right)dt_1\cdots dt_{r-2}
   \end{multlined}\notag\\
   &=\begin{multlined}[t]
     \frac{\Gamma(s_1+\cdots+s_r-1)}{(2\pi i)^{r-2}\Gamma(s_3)\cdots\Gamma(s_r)}
     \int_{\mathcal{L}_1}\cdots\int_{\mathcal{L}_{r-2}}\frac{\Gamma(1-s_1+t_1+\cdots+t_{r-2})}{\Gamma(s_2+\cdots+s_r+t_1+\cdots+t_{r-2})}\\\times\left(\prod_{j=1}^{r-2}\Gamma(s_{j+2}+t_j)\Gamma(-t_j)\right)\\\times(-1)^{t_1+\cdots+t_{r-2}}\zeta_{A,r-1}(s_1+s_2-1,s_3+t_1,\ldots,s_{r-1}+t_{r-3},s_r+t_{r-2};-t_1,\ldots,-t_{r-2})dt_1\cdots dt_{r-2}.
   \end{multlined}\notag
   \end{align}
   Applying \cite[Theorem 7.8]{Y.K.H.2}, the zeta functions of root systems $\zeta_r(\bm{z};\Delta)$ is bounded by \[O((\text{polynomials in } z_i)e^{\theta_i|\Im z_i|}),\quad|\theta_i|<\frac{\pi}{2}\] in terms of $z_i$ and for $1\le i\le n$. This ensures the convergence of the integral. Hence, by Lemma \ref{lem2} we complete the proof.
 \end{proof}
 \begin{rem}
   The author has also found a self-contained proof without using the existing theory of zeta functions of root systems. This is obtained by applying Taylor's theorem to the integrand of $\Psi_{r-1}(s_1,\ldots,s_r;n_1,\ldots,n_1+\cdots+n_{r-1};0)$. However, this method yields analytic continuation only over the whole $\mathfrak{A}_r$.
 \end{rem}
Applying the previous facts, we give the proof of Theorem \ref{main theorem}.
 
\begin{proof}[Proof of Theorem \ref{main theorem}]
\phantom{a}\\
Changing $(s_1,s_2,s_3,\ldots,s_r)$ by $(1-\textup{wt}(\bm{s})+s_1,1-\textup{wt}(\bm{s})+s_2,s_3,\ldots,s_r)$ in Theorem \ref{Theorem3} and applying equation (\ref{label5}), we have 
\begin{align}
   &\mathscr{G}_r(1-\textup{wt}(\bm{s})+s_1,1-\textup{wt}(\bm{s})+s_2,s_3,\ldots,s_r)\notag\\&=\begin{multlined}[t]\Gamma(\textup{wt}(\bm{s})-s_1)\sum_{k_1,\ldots,k_{r-1}=1}^\infty \sigma_{EZ,r-1}(2\textup{wt}(\bm{s})-s_1-s_2-1,-s_3,\ldots,-s_r;k_1,\ldots,k_{r-1})\\\times\sum_{m_1=0}^\infty\frac{(s_3)_{m_1}(1-\frac{k_1}{k_1+k_2})^{m_1}}{m_1!}\cdots\sum_{m_{r-2}=0}^\infty\frac{(s_r)_{m_{r-2}}(1-\frac{k_1}{k_1+\cdots+k_{r-1}})^{m_{r-2}}}{m_{r-2}!}\\\times(\textup{wt}(\bm{s})-s_1)_{m_1+\cdots+m_{r-2}}\Psi(\textup{wt}(\bm{s})-s_1+m_1+\cdots+m_{r-2},\textup{wt}(\bm{s});2\pi ik_1;1)\\+\Gamma(\textup{wt}(\bm{s})-s_1)\sum_{k_1,\ldots,k_{r-1}=1}^\infty\sigma_{EZ,r-1}(2\textup{wt}(\bm{s})-s_1-s_2-1,-s_3,\ldots,-s_r;k_1,\ldots,k_{r-1})\\\times\sum_{m_1=0}^\infty\frac{(s_3)_{m_1}(1-\frac{k_1}{k_1+k_2})^{m_1}}{m_1!}\cdots\sum_{m_{r-2}=0}^\infty\frac{(s_r)_{m_{r-2}}(1-\frac{k_1}{k_1+\cdots+k_{r-1}})^{m_{r-2}}}{m_{r-2}!}\\\times(\textup{wt}(\bm{s})-s_1)_{m_1+\cdots+m_{r-2}}\Psi(\textup{wt}(\bm{s})-s_1+m_1+\cdots+m_{r-2},\textup{wt}(\bm{s});-2\pi ik_1;1).\end{multlined}\label{label8}
\end{align}
By Theorem \ref{Theorem3} and (\ref{label8}), we obtain Theorem \ref{main theorem}.
\end{proof}
\begin{rem}
 The formula (\ref{label8}) can be regarded as an analogue of \cite[Theorem 2.1]{Y.K.H.1}.
\end{rem}
 
\end{document}